\documentclass[11pt,twoside]{amsart}
\usepackage{amsfonts}
\usepackage{amssymb}
\usepackage{amscd}
\input xy
\xyoption{all}
\title[Generic multiplication]{A generic multiplication in quantised Schur algebras}
\author{Xiuping Su}

\newtheorem{theorem}{Theorem}[section]
\newtheorem{lemma}[theorem]{Lemma}
\newtheorem{proposition}[theorem]{Proposition}
\newtheorem{definition}[theorem]{Definition}

\newtheorem{corollary}[theorem]{Corollary}
\newtheorem{remark}[theorem]{Remark}

\newcommand{\rep}{\mathrm{Rep}}

\renewcommand{\hom}{\mathrm{Hom}}
\newcommand{\gl}{\mathrm{GL}}

\newcommand{\dv}[1]{{\bf #1}}

\newcommand{\field}{{k}}

\renewcommand{\ker}{\mathrm{Ker}}
\newcommand{\im}{\mathrm{Im}}

\newcommand{\lra}{\longrightarrow}

\newcommand{\ra}{\rightarrow}
\newcommand{\sdp}{\times\kern-.2em\vrule height1.1ex depth-.05ex}
\newcommand{\epi}{\lra \kern-.8em\ra}

\normalbaselines
\begin{document}

\begin{abstract}
We define a generic multiplication in quantised Schur algebras and thus obtain a new algebra structure
in the Schur algebras. We prove that via a modified version of the map from quantum groups
to quantised Schur algebras, defined in \cite{BLM}, a subalgebra of this new algebra is a 
quotient of the monoid algebra
in Hall algebras studied in \cite{Reineke}.  We also prove that the subalgebra of the new algebra gives a
geometric realisation of a positive part of $0$-Schur algebras, defined in \cite{Donkin}. Consequently,
we obtain a multiplicative basis for the positive part of $0$-Schur algebras.
\end{abstract}

\maketitle

\section*{Introduction}

Schur algebras $S(n, r)$ were invented by I. Schur to classify the polynomial representations of
the complex general linear group $\mathrm{Gl}_n(\mathbb{C})$. Quantised Schur algebras $S_q(n, r)$ are
quantum analogues of Schur algebras. Both quantised Schur algebras $S_q(n, r)$
and classical Schur algebras $S(n, r)$
have applications to the representation
theory of $ \mathrm{Gl}_n$ over fields of undescribing characteristics.

In \cite{BLM} A. A. Beilinson, G. Lusztig and R. MacPherson gave a geometric construction of quantised
enveloping algebras of type $\mathbb{A}$. Among other important results they defined surjective algebra
homomorphisms $\theta$, from the integral form of the quantised enveloping algebras to certain
finite dimensional associative algebras.
 They first
defined a multiplication of pairs of $n$-step partial flags in a vector space $k^r$ over a finite field $k$ and
thus obtained a finite dimensional associative algebra.  They also
studied how the structure constants behave when $r$ increases by a multiple of $n$. Then,
by taking a certain limit they obtained the quantised enveloping algebras of type $\mathbb{A}$.
J. Du remarked in \cite{Duj} that the finite dimensional associative algebras studied in \cite{BLM} are canonicallly
isomorphic to the quantised Schur algebras studied by R. Dipper and G. James in \cite{DJ}.

The aim of this paper is to study a generic version of the multiplication of pairs of partial flags defined
in \cite{BLM}. By this generic multiplication we get an algebra structure in the quantised Schur algebras. We prove that
a certain subalgebra of this new algebra is a quotient of the monoid algebra in Hall algebras studied by M. Reineke
in \cite{Reineke}. Via
a modified version of the surjective algebra homomorphism $\theta$, defined in \cite{BLM}, we prove that the
subalgebra is isomorphic to a positive part of $0$-Schur algebras, studied by S. Donkin in \cite{Donkin}.
Thus we achieve a geometric construction of the positive parts of the $0$-Schur algebras.

This paper is organized as follows. In Section 1 we recall definitions and results  in \cite{BLM}
on the multiplication of pairs of partial flags (see also
\cite{Duj, GreenR}). In Section 2 we recall definitions and results on the monoid
given by generic extensions studied in \cite{Reineke}. In Section 3 we study a generic version
of the multiplication of pairs of partial flags  in \cite{BLM}, and prove that this generic multiplication
gives us a new algebra structure in quantised Schur algebras. In Section 4 we prove results on connection between
our new algebras and the monoid algebras given by generic extensions in \cite{Reineke} and to $0$-Schur algebras.
We also provide a multiplicative basis for a positive part of $0$-Schur algebras. As a remark, we would like to
mention that this multiplicative basis is related to Lusztig's canonical basis.

\section{$q$-Schur algebras as quotients of quantised enveloping algebras}

In this section we recall some definitions and results from \cite{BLM} on $q$-Schur algebras as quotients
of quantised enveloping algebras (see also \cite{Duj, GreenR}).

\subsection{q-Schur algebras}
Denote by $\Theta_r$ the set of $n\times n$ matrices whose entries are non-negative integers and sum to $r$. Let
$V$ be an $r$-dimensional vector space over a field $k$. Let $\mathcal{F}$ be the set of all $n$-steps
flags in $V$:
$$V_1\subseteq V_2\subseteq \cdots\subseteq V_n=V. $$
The group $\gl(V)$ acts naturally by change of basis on $\mathcal{F}$.  We let $\gl(V)$ act
diagonally on $\mathcal{F}\times\mathcal{F}$. Let
$(f, f')\in \mathcal{F}\times\mathcal{F}$, we write $$f= V_1\subseteq V_2\subseteq \cdots\subseteq V_n=V \mbox{ and }
f'= V'_1\subseteq V'_2\subseteq \cdots\subseteq V'_n=V.$$
Let $V_0=V_0'=0$ and define
$$a_{ij}=\dim({V_{i-1}+V_i\cap V'_{j}})-\dim({V_{i-1}+V_i\cap V'_{j-1}}). $$
Then the map  $(f, f')\mapsto (a_{ij})_{ij}$ induces a bijection between the set of $\gl(V)$-orbits in
$\mathcal{F}\times\mathcal{F}$ and the
set $\Theta_r$. We denote by $\mathcal{O}_A$ the $\gl(V)$-orbit in $\mathcal{F}\times \mathcal{F}$
corresponding to the matrix $A\in \Theta_r$.

Now suppose that $\field$ is a finite field with $q$ elements. Let $A, \;A', \;A''\in \Theta_r$ and let
$(f_1, f_2)\in \mathcal{O}_{A''}$. Following Proposition 1.1 in \cite{BLM},
 there exists a polynomial $g_{A, A', A''}= c_0+c_1q\cdots+c_mq^m$, given by
$$g_{A, A', A''}= |\{ f\in \mathcal{F}| (f_1, f)\in \mathcal{O}_A, (f, f_2)\in \mathcal{O}_{A'} \}|,$$

\noindent where $c_i$ are integers that do not depend on $q$, the cardinality of the field $k$,
and $(f_1, f_2)\in \mathcal{O}_{A''}$.

Now recall that the $q$-Schur algebra $S_q(n, r)$ is the free $\mathbb{Z}[q, q^{-1}]$-module with basis
$\{e_A| A\in \Theta_r\}$, and with an associative multiplication given by

$$ e_A e_{A'} =\sum_{A''\in \Theta_r}g_{A, A', A''}e_{A''}.
$$

For a matrix $A\in  \Theta_r$, denote by $\mathrm{ro}(A)$ the vector $(\sum_ja_{1j}, \sum_ja_{2j},\cdots,
\sum_ja_{nj} )$ and by $\mathrm{co}(A)$ the vector $(\sum_ja_{j1}, \sum_ja_{j2},\cdots,
\sum_ja_{jn} )$. By the definition of the multiplication it is easy to see that
$$e_Ae_{A'}=0 \mbox{ if } \mathrm{co}(A)\not=\mathrm{ro}(A').$$

Denote by $E_{ij}$ the elementary $n\times n$ matrix with $1$ at the entry $(i, j)$ and $0$ elsewhere.
We recall a lemma, which we will use later,   on the multiplication defined above.

\begin{lemma}[\cite{BLM}] \label{Lemma3.2}
Assume that $1\leq h<n$. Let $A=(a_{ij})\in \Theta_r$.
Assume that $B=(b_{ij})\in \Theta_r$ such that $B-E_{h, h+1}$ is a diagonal matrix and
$\mathrm{co}(B)=\mathrm{ro}(A)$. Then
$$
e_Be_A=\sum_{p: a_{h+1, p}>0}v^{2\sum_{j>p}a_{hj}}\frac{v^{2(a_{hp}+1)}-1}{v^2-1}e_{A+E_{h,p}-E_{h+1, p}}.
$$
\end{lemma}

\subsection{The map $\theta: U_\mathcal{A}(\mathrm{gl}_n)\rightarrow S_v(n, r)$}
Let $\mathcal{A}=\mathbb{Z}[v, v^{-1}]$ and $v^2=q$. Let $U_\mathcal{A}(\mathrm{gl}_n)$ be the integral
form of the quantised enveloping algebra of the Lie algebra $\mathrm{gl}_n$.
Denote by $S_v(n, r)$
the algebra  $\mathcal{A}\otimes S_q(n, r)$.
Let $$\theta:  U_\mathcal{A}(\mathrm{gl}_n)\rightarrow S_v(n, r)$$
be the surjective algebra homomorphism defined by A. A. Beilinson, G. Lusztig and R. MacPherson in \cite{BLM}.
Through the map $\theta$ we can view the Schur algebra $S_v(n, r)$ as a quotient of the
quantised enveloping algebra $U_\mathcal{A}(\mathrm{gl}_n)$.
We are interested in the restriction of $\theta$ to the positive part $U^+$
of  $U_\mathcal{A}(\mathrm{gl}_n)$.

Unless stated otherwise, we
let $Q$ be the linearly oriented quiver of type $\mathbb{A}_{n-1}$:
$$\xymatrix{Q: 1\ar[r]&2\ar[r]&\cdots\ar[r]&n-1}.$$
By a well-known result of C. M. Ringel
(see \cite{RingelBanach, RingelInv}), the algebra $U^+$ is
isomorphic to the twisted Ringel-Hall algebra $H_q(Q)$, which is generated
by isomorphism classes of simple $kQ$-modules via Hall multiplication.

Denote by $S_i$ the simple module of the path algebra $kQ$ associated to vertex $i$ of $Q$.
By abuse of notation we also denote by $M$ the isomorphism class of a $kQ$-module
$M$. For any $s\in \mathbb{N}$, denote by $D_s$ the set of diagonal matrices satisfying
that the entries are non-negative integers and that the sum of the entries is $s$.
For a matrix $A\in \Theta_r$, denote by $[A]= v^{-\dim\mathcal{O}_A+\dim pr_1(\mathcal{O}_A )}e_A$,
where $pr_1$ is the natural projection to the first component of
$\mathcal{F}\times \mathcal{F}$.
Now the map $\theta$ can be defined on the twisted Ringel-Hall algebra as follows:
$$ \theta: H_q(Q)\rightarrow S_v(n, r), \;\; S_i\mapsto \sum_{D\in D_{r-1}}[E_{i, i+1}+D].
$$

\section{A monoid given by generic extensions} \label{monoidreineke}

In this section we briefly recall definitions and results on
the monoid of generic extensions in \cite{Reineke}, and we
let $k$ be an algebraically closed field.
Results in this
section work for any Dynkin quiver $Q=(Q_0, \; Q_1)$, where $Q_0=\{1, \cdots, n\}$ is the set
of vertices of $Q$ and $Q_1$ is the set of arrows of $Q$.
We denote by $\mathrm{mod} kQ$ and $\rep(Q)$,
respectively, the  category of finitely generated left $kQ$-modules and the category of finite dimensional
representations of $Q$. We don't distinguish a representation of $Q$ from the corresponding $kQ$-module.

Let $\dv b\in \mathbb{N}^{n-1}$. Denote by
$$\rep(\dv b)= \Pi_{i\rightarrow j\in Q_1}\hom_k(k^{b_i}, k^{b_j})$$
the representation variety of $Q$,
which is an affine space
consisting of representations with dimension vector $\dv b$. The group $\gl(\dv b)=\Pi_i\gl(b_i)$
acts on $\rep(\dv b)$ by conjugation and there is a one-to-one correspondence between
$\gl(\dv b)$-orbits in $\rep(\dv b)$ and isomorphism classes of representations in $\rep(\dv b)$.
Denote by $\mathcal{E}(M, N)$ the
subset of $\rep(\dv b)$, containing points which are extensions of $M$ by $N$.

\begin{lemma}[\cite{Reineke}]
 The set $\mathcal{E}(M, N)$ is an irreducible subset of $\rep(\dv b)$.
\end{lemma}

Thus there exists a unique open $\gl(\dv b)$-orbit in $ \mathcal{E}(M, N)$. We say a point
in $\mathcal{E}(M, N)$ is generic if it is contained in the open orbit. Generic points in
$\mathcal{E}(M, N)$ are also called generic extensions of $M$ by $N$.

\begin{definition}[\cite{Reineke}]
Let $M$ and $N$ be two isomorphism classes in $\mathrm{mod} kQ$. Define a multiplication
$$M\ast N=G,$$
where $G$ is the isomorphism class of the generic points in $\mathcal{E}(M, N)$.
\end{definition}

Denote by ${\bf H}_q(Q)$ the Ringel-Hall algebra  over $\mathbb{Q}[q]$ and by
${\bf H}_0(Q)$ the specialisation of ${\bf H}_q(Q)$ at $q=0$.

\begin{theorem}[\cite{Reineke}]\label{hallmonoid}
(1) $\mathcal{M}=(\{M|M$ is an isomorphism class in $ \mod kQ\}, \; \ast)$ is a monoid.

(2) $\mathbb{Q}\mathcal{M}\cong {\bf H}_0(Q)$ as algebras.
\end{theorem}

\section{A monoid given by a generic multiplication in $q$-Schur algebras}

In this section we define a generic multiplication in the $q$-Schur algebra
$S_q(n, r)$. Via this multiplication we obtain a new algebra in $S_q(n, r)$.
We let $k$ be an algebraically closed field in this section.

Denote by
 $$\Theta^u_r=\{A\in \Theta_r| A \mbox{ is an upper triangular matrix}\}.$$
Let $A, A'\in \Theta_r^u$, define
$$\mathcal{E}(A, A')=\{(f_1, f_2)\in \mathcal{F}\times\mathcal{F}| \exists f \mbox{ such that }
(f_1, f)\in \mathcal{O}_A \mbox{ and } (f, f_2)\in \mathcal{O}_{A'}\}.
$$
We denote by $M(ij)$ the indecomposable representation of $Q$ with dimension vector
$\sum_{l=i}^{j-1}\dv e_l$, where $\dv e_l$ is the simple root of $Q$ associated to vertex $i$,
that is, $\dv e_i$ is the dimension vector of the simple module $S_i$. By the module determined by
a matrix $A\in \Theta_r^u$ we mean the module $\bigoplus_{i<j} M(ij)^{a_{ij}}$.
Note that for any $(f_1, f_2)\in \mathcal{E}(A, A')$, we have $f_2$ is a subflag of $f_1$.
Note  also that by omitting the last step of a partial flag $f$ in $\mathcal{F}$, we can view
$f$ as a projective $kQ$-module and by abuse of notation we still denote the projective
module by $f$.
Now for $A\in \Theta_r^u$, suppose that $(f, h)\in \mathcal{O}_A$ and that
$M$ is the module determined by $A$.  We have a short exact sequence
$$\xymatrix{0\ar[r] &h\ar[r]&f \ar[r] & M\ar[r]&0}, $$
that is, $h\subseteq f $ is a  projective resolution of $M$.

Let $\dv b\in \mathbb{N}^{n-1}$ and  denote
by $k^{\dv b}$ the $Q_0$-graded vector space  with $k^{b_i}$ as its $i$-th homogeneous component,
where $i$ is a vertex of $Q$.
Denote by $\hom_{\mathrm{gr}}(f, k^{\dv b})$ the set of graded linear maps between $f$  and $k^{\dv b}$,
where $f$ is a partial flag in $\mathcal{F}$ viewed as a $Q_0$-graded
vector space by omitting its last step.

\subsection{Relation between generic points in $\mathcal{E}(A, A')$ and in $\mathcal{E}(M, N)$}
For $A, A'\in \Theta_r$, we write $A\leq A'$ if $ \mathcal{O}_{A'}$ is contained
in the Zariski closure of $\mathcal{O}_{A}$. In this case we say that $(f_1, f_2)\in \mathcal{O}_{A}$
degenerates to $(f'_1, f'_2)\in \mathcal{O}_{A'}$. Lemma 3.7 in \cite{BLM} implies the
existence of generic points in $\mathcal{E}(A, A')$, in the sense that the closure of their orbit contains
orbits of all the other points in $\mathcal{E}(A, A')$.
That is, there is a unique open orbit in $\mathcal{E}(A, A')$.
In this subsection we will show that there is a nice correspondence between generic points $\mathcal{E}(A, A')$ and generic
points in subset $\mathcal{E}(M, N)$, where $M$ and $N$ are the modules determined by
$A$ and $A'$, respectively.







Let $(f_1, f_2)\in \mathcal{F}\times \mathcal{F}$ with $f_2$ a subflag of $f_1$.
Denote by $\dv a_1$ and $\dv a_2$, respectively, the dimension vectors
of the projective modules $f_1$ and $f_2$. Let $\dv b=\dv a_1-\dv a_2$.
We define some sets as follows.
$$\mathrm{Inj}(f_2, f_1)=\{\sigma\in \hom_{\field Q}(f_2, f_1)|\sigma \mbox{ is injective}\};$$
$$\mathcal{S}_1=\{(\sigma, \eta)\in \mathrm{Inj}(f_2, f_1)\times  \hom_{\mathrm{gr}}(f_1, k^{\dv b})|
\eta\mbox{ is surjective and } \eta\sigma=0\};
$$
$$\mathcal{S}'_1=\{(\sigma, \eta)\in \hom_{kQ}(f_2, f_1)\times  \hom_{\mathrm{gr}}(f_1, k^{\dv b})|
\eta\mbox{ is  surjective, } $$$$
\mbox{ker}\eta \mbox{ is a } kQ\mbox{-module and }
\eta\sigma=0\};
$$
$$\mathcal{S}_2=\{\eta \in \hom_{\mathrm{gr}}(f_1, k^{\dv b})|
\eta \mbox{ is surjective and } \mbox{ker}\eta \mbox{ is a } kQ\mbox{-module}\};
$$
$$\mathcal{S}'_2=\{(M, \eta)\in \rep(\dv b)\times \hom_{\mathrm{gr}}(f_1, k^{{\bf b}})|
\eta: f_1\rightarrow M \mbox{ is a } kQ\mbox{-homomorphism}\};
$$
$$
\mbox{Inj}_{A, A'}(f_2, f_1)=\{\sigma\in \mbox{Inj}(f_2, f_1)|\mbox{cok}(\sigma)\in \mathcal{E}(M, N) \},
$$
where $M$ and $N$ are the modules determined by $A$ and $A'$, respectively.

For convenience we denote by $\mathrm{Inj}(f_2, f_1)$ by $\mathcal{S}_3$.
We obtain some fibre bundles as follows.

\begin{lemma}[\cite{JS}]\label{bundle1}
The natural projection $\pi_1: \mathcal{S}'_1\rightarrow \mathcal{S}_2$ is a vector bundle.
\end{lemma}

\begin{lemma}[\cite{JS}]\label{bundle2}
The natural projection $\pi_2: \mathcal{S}_1\rightarrow \mathcal{S}_3$ is a principal $\gl(\dv b)$-bundle.
\end{lemma}

\begin{lemma}[\cite{JS}]\label{bundle3}
The natural projection $\pi_3: \mathcal{S}'_2\rightarrow \rep(\dv b)$ is  a vector bundle.
\end{lemma}

Note that  any $\eta\in S_2$ determines a unique module $M\in \rep(\bf b)$ and this defines
an open embedding of $\mathcal{S}_2$ into $\mathcal{S}_2'$. So we can view
$\mathcal{S}_2$ as an open subset of $\mathcal{S}_2'$.

\begin{lemma}
(1) $\mathrm{Inj}_{A, A'}(f_2, f_1)=\pi_2(\mathcal{S}_1\cap \pi^{-1}_1(\mathcal{S}_2\cap \pi^{-1}_3(\mathcal{E}(M, N)))).$

(2) $\mathrm{Inj}_{A, A'}(f_2, f_1)$ is irreducible.
\end{lemma}

\begin{proof}
Following the definitions of $\pi_1$ and $\pi_3$, $\mathrm{cok}\sigma\in \mathcal{E}(M, N)$ for
any $(\sigma, \eta)\in \mathcal{S}_1\cap \pi^{-1}_1(\mathcal{S}_2\cap \pi^{-1}_3(\mathcal{E}(M, N)))$.
Therefore $\sigma=\pi_2((\sigma, \eta))\in  \mathrm{Inj}_{A, A'}(f_2, f_1)$. On the other hand,
suppose that $\sigma \in \mathrm{Inj}_{A, A'}(f_2, f_1)$. Then
$\sigma\in  \pi_2(\mathcal{S}_1\cap \pi^{-1}_1(\mathcal{S}_2\cap \pi^{-1}_3(X)))$, where
$X$ is the module determined by $\eta$ for a preimage $(\sigma, \eta)\in \pi_2^{-1}(\sigma)$.
This proves (1). Now (2) follows from (1) and Lemmas \ref{bundle1}-\ref{bundle3}.
\end{proof}

Let $$\mathcal{E}'(A, A')=\{(f_1, f)\in \mathcal{E}(A, A')|f\in \mathcal{F}\}.$$
Define $$\pi: \mathrm{Inj}_{A, A'}(f_2, f_1)\rightarrow \mathcal{F}\times \mathcal{F}, \; \sigma\mapsto (f_1, \mathrm{Im \sigma}),$$
where $\mathrm{Im}\sigma$ can be viewed as a flag in $\mathcal{F}$ with the last step the natural embedding of
$\mathrm{Im}\sigma(f_1)_{n-1}$ into $V$.

\begin{lemma}\label{irred}
(1) $\mathrm{Im}\pi=\mathcal{E}'(A, A')$.

(2) $\mathcal{E}'(A, A')$ is irreducible.
\end{lemma}

\begin{proof}
Let $\sigma\in  \mathrm{Inj}_{A, A'}(f_2, f_1)$. Then by the following diagram,
$$
\xymatrix{
f_2\ar[d]\ar[r]& f_2\ar[d]^{\sigma}\ar[r]& 0\ar[d]\\
\mathrm{ker}\xi\eta\ar[r]\ar[d] & f_1 \ar[r]^{\xi\eta}\ar[d]^\eta&M, \ar[d]\\
N\ar[r]&\mathrm{cok}\sigma\ar[r]_\xi&M
}
$$
where each square commutes and all rows and columns are short exact sequences, we know that
$(f_1, \im\sigma)\in \mathcal{E}'(A, A')$. On the other hand by the definition of $ \mathcal{E}'(A, A')$,
for any $(f_1, f)\in \mathcal{E}'(A, A')$, the natural embedding $f_2\cong f\subseteq f_1$ is in
$\mathrm{Inj}_{A, A'}(f_2, f_1)$. Now the irreducibility of $ \mathcal{E}'(A, A')$ follows from
that of $ \mathrm{Inj}_{A, A'}(f_2, f_1)$.
\end{proof}

As a consequence of Lemma \ref{irred}
we can see the existence of a unique dense open orbits in $\mathcal{E}(A, A')$. Indeed,
by Lemma \ref{irred} and the surjective map $ \mathcal{E}'(A, A')\times \gl(r)\rightarrow \mathcal{E}(A, A')$,
the set $\mathcal{E}(A, A')$ is irreducible.
Since there are only finitely many $\gl(r)$-orbits in $\mathcal{E}(A, A')$, there exists
a unique dense open orbits in $\mathcal{E}(A, A')$.

\begin{definition} Let $\mathcal{O_{A''}}$ be the dense open orbit in $\mathcal{E}(A, A')$.
We say that an injection $\sigma:f'\rightarrow f$ is generic in
$\mathrm{Inj}_{A, A'}(f', f)$ if the pair of flags $(f,\; \im \sigma)$ is contained in $\mathcal{O_{A''}}$.
\end{definition}

\begin{proposition}\label{genericprop}
Let $(\sigma, \eta)\in\mathcal{S}_1$, $(f_1, f)\in \mathcal{O}_{A'}$ and
$(f, f_2)\in \mathcal{O}_{A''}$. Then $\sigma$ is generic in $\mathrm{Inj}_{A, A'}(f_2, f_1)$
if and only if the module determined
by $\eta$ is generic in $\mathcal{E}(M, N)$, where $M$ and $N$ are the modules determined by
$A$ and $A'$, respectively.
\end{proposition}

\begin{proof}
Suppose that $\mathcal{O}_{A''}$ is the dense open orbit in $ \mathcal{E}(A, A')$ and that
$ \mathcal{O}_X$ is the dense open orbit in $ \mathcal{E}(M, N)$.
By Lemmas \ref{bundle1}-\ref{bundle3},
$\pi_2(\mathcal{S}_1\cap \pi_1^{-1}(\mathcal{S}_2\cap\pi_3^{-1}(\mathcal{O}_X)))$
is open in $\mathrm{Inj}_{A, A'}(f_2, f_1)$. By Lemma \ref{irred}, $\pi^{-1}(\mathcal{O}_{A''}\cap \mathcal{E}'(A, A') )$
is open in $\mathrm{Inj}_{A, A'}(f_2, f_1)$. Since  $\mathrm{Inj}_{A, A'}(f_2, f_1)$ is irreducible, the
intersection
$\pi_2(\mathcal{S}_1\cap \pi_1^{-1}(\mathcal{S}_2\cap\pi_3^{-1}(\mathcal{O}_X)))
\cap \pi^{-1}(\mathcal{O}_{A''}\cap \mathcal{E}'(A, A') )$
is non-empty. Therefore, $X$ is the module determined by $A''$. This finishes the proof.
\end{proof}

\subsection{A generic multiplication in $S_q(n, r)$} We
 now define a multiplication, called a generic multiplication, by

$$e_A\circ e_{A'}= \left\{\begin{tabular}{ll}$e_{A''}$ & if  $\mathcal{E}(A, A')\not=\emptyset$,\\
$0$ &otherwise,
\end{tabular}\right.$$
where $\mathcal{O}_{A''}$ is the dense open orbit in $\mathcal{E}(A, A')$.

\begin{proposition}\label{associativity}
Let $A$, $A'$, $A''\in \Theta_r^u$. Then $(e_A\circ e_{A'})\circ e_{A''}=e_A\circ(e_{A'}\circ e_{A''})$
\end{proposition}

\begin{proof}
By the definition of the multiplication $\circ$, we see that $(e_A\circ e_{A'})\circ e_{A''}=0$ implies
that $e_A\circ(e_{A'}\circ e_{A''})=0$ and vice versa. So we may assume that neither of them is zero.
Let $M$, $N$, $L$ be the module determined by $A$, $A'$, $A''$, respectively. By Lemma 3.1 in \cite{Reineke}, we
know that $(M\ast N)\ast L=M\ast(N\ast L)$. Now the proof follows from Proposition \ref{genericprop}.
\end{proof}

We can now state the main result of this section.

\begin{theorem}\label{algebrastructure}
$\mathbb{Q}(\{e_A|A\in \Theta_r^u\}, +, \circ)$ is an algebra with unit $\sum_{D\in D_r}e_D$.
\end{theorem}

\begin{proof}
We need only to show that $\sum_{D\in D_r}e_D$ is the unit. Let $D$ be a diagonal matrix and let
$(f_1, f_2)\in \mathcal{O}_D$. Then $f_1=f_2$. For any $A\in \Theta_r^u$, by the definition of
the generic mulitiplication,
$$e_A\circ \sum_{D\in D_r}e_D=e_A \circ e_C.
$$
where $C$ is the diagnal matrix $\mathrm{diag}(\sum_ja_{1j}, \cdots, \sum_ja_{nj})$.
Since $e_Ae_C=\sum_{B}g_{A, C, B}e_B$, where $(f_1, f)\in \mathcal{O}_B$ and $g_{A,  C, B}=|\{f|(f_1, f)\in \mathcal{O}_A,
(f, f)\in \mathcal{O}_C\}|$,  we see that $e_Ae_C=e_A$.
Therefore $e_A\circ e_C=e_A$.
Similarly, $(\sum_{D\in D_r}e_D)\circ e_A=e_A$. Therefore $\sum_{D\in D_r}e_D$ is the unit.
\end{proof}

We denote the algebra $\mathbb{Q}(\{e_A|A\in \Theta_r^u\}, +, \circ)$ in Theorem \ref{algebrastructure}
by $S_0^+$.

\section{The algebra $S_0^+$ as a quotient and 0-Schur algebras}
We have two tasks in this section. We will first prove that a certain subalgebra of $S_0^+$
is a quotient of the monoid algebra defined in Section \ref{monoidreineke}. We will then prove that this
subalgebra gives a geometric realisation of a positive part of 0-Schur algebras.

It is well-known that the specialisation of $S_q(n, r)$ at $q=1$ gives us the classical Schur algebra
$S(n, r)$ of type $\mathbb{A}$. Much about the structure and representation theory of $S(n, r)$ is known,
see for example \cite{GreenJ}. A natural question is to consider the specialisation of $S_q(n, r)$
at $q=0$, which is called $0$-Schur algebra and denoted by $S_0(n, r)$. $0$-Schur algebras have been
studied in \cite{Donkin, KrobTh, Solomon}. In this section the $0$-Schur algebras will be studied  from a
different point of view, that is, via a modified version $\theta: {\bf H}_q(Q) \rightarrow S_q(n, r)$ of the map
$\theta: U_\mathcal{A}(\mathrm{gl}_n)\rightarrow S_v(n, r)$.

We call $\theta({\bf H}_q(Q))$ the positive part of the $q$-Schur algebra, and denote it by
$S^+_q(n, r)$.  Denote the specialisation of $S_q^+(n, r)$ at $q=0$ by $S_0^+(n, r)$.
Denote by $S_0^{++}$ the subalgebra of $S_0^+$, generated by $l_{A,r}=\sum_De_{A+D}$, where $A$ is a strict upper
triangle matrix with its entries non-negative intergers and the sum is taken over all diagnal matrices
in $D_{r-\sum_{i, j}a_{ij}}$.

\subsection{A modified version of $\theta$}

 For convenience we denote
by $E_i$ the element $l_{E_{i, i+1}, r}$ in $S_q(n, r)$. We have the following result.

\begin{proposition}\label{serre}
The elements $E_1, \dots, E_{n-1}$ satisfy the following modified quantum Serre relations:
$$
E_i^2E_j-(q+1)E_iE_jE_i+qE_jE_i^2=0 \mbox{ for } |i-j|=1 \mbox{ and }
$$
$$
E_iE_j-E_jE_i=0 \mbox{ for } |i-j|> 1.
$$
\end{proposition}

\begin{proof} We only prove the first equation for $j=i+1$. The remaining part can be done in a similar way.
By Lemma \ref{Lemma3.2}, we have the following.

\begin{center}
\begin{tabular}{l}
$E_iE_{i+1}=l_{E_{i, i+2}, r}+l_{E_{i, i+1}+E_{i+1, i+2}, r}$,\\
$E_{i+1}E_{i}=l_{E_{i, i+1}+E_{i+1, i+2}, r} $,\\
$E_{i}E_{i}=(q+1)l_{2E_{i, i+1}, r}$,\\
$E_{i}l_{E_{i, i+2}, r}=ql_{E_{i, i+1}+E_{i, i+2}, r},$\\
$E_{i}l_{E_{i, i+1}+E_{i+1, i+2}, r}=l_{E_{i, i+1}+E_{i, i+2}, r}+(q+1)l_{2E_{i, i+1}+E_{i+1, i+2}, r}$,\\
$E_{i+1}E_{i}E_{i}=(q+1)l_{2E_{i, i+1}+E_{i+1, i+2}, r}$.
\end{tabular}
\end{center}
Therefore,
$$
E_i^2E_{i+1}-(q+1)E_iE_{i+1}E_i+qE_{i+1}E_i^2=E_{i}(E_{i}E_{i+1}-(q+1)E_{i+1}E_{i})+qE_{i+1}E_i^2=0.
$$
\end{proof}

Following Proposition \ref{serre} and \cite{RingelInv}, we can now modify the restriction $\theta|_{\bf{H}_q(Q)}$ as follows.
$$
\theta: {\bf H}_q(Q)\rightarrow S_q(n, r), $$ $$ S_i\mapsto E_i
$$
From now on, unless stated otherwise, by $\theta$ we mean the modified map $\theta|_{{\bf H}_q(Q)}$.
The following result is a modified version of Proposition 2.3 in \cite{GreenR} and we will give
a direct proof.  For any two modules $M, N$, recall that Hall multiplication of $M$ and $N$ is given
by $$M N =\sum_XF^X_{MN}X,$$
where $F^X_{MN}=|\{U\subseteq X|U\cong N, X/U\cong N\}|$ and where the sum is taken over all the isomorphism
classes of modules.

\begin{proposition}
Let $A$ be a strict upper triangular matrix with entries non-negative integers
and let $M$ be the module determined by $A$.  Then
$$
\theta(M)=\left\{\begin{tabular}{ll} $l_{A, r}$& if  $\sum_{i, j}a_{ij}\leq r$,  \\
$0$ & otherwise. \end{tabular}\right.
$$
\end{proposition}
\begin{proof}
Note that $\mathcal{B}=\{\Pi_{i, j}M(ij)^{x_{ij}}|x_{ij}\in \mathbb{Z}_{\geq 0} \}$
is a PBW-basis of ${\bf H}_q(Q)$, where
the product is ordered as follows:  $M(ij)$ is on the left hand side of $M(st)$ if
either $i=s$ and $j>t$, or $i>s$. Let $M\in \mathcal{B}$ be
the module determined by $A$. Supose that $M(st)$ is the left most term with $x_{st}>0$.
We can write $M= M(st)\oplus M'$.

First consider the case $M=M(st)$, that is, $M$ is indecomposable. We may suppose that $t-s>1$.
Then $M=M(s, t-1)M(t-1, t)-M(t-1, t)M(s, t-1)$. By induction on the lenght of $M$, Lemma \ref{Lemma3.2}
and a dual version of it, we have
$$\theta(M(s, t-1)M(t-1, t))=l_{E_{s, t}, r}+l_{E_{s,t-1}+E_{t-1, t} , r},
$$
$$\theta(M(t-1, t)M(s, t-1))=l_{E_{s,t-1}+E_{t-1, t} , r}.
$$
Therefore $\theta(M)=\theta(M(s, t-1)M(t-1, t))-\theta(M(t-1, t)M(s, t-1))= l_{E_{st}, r}$.

Now consider the case that $M$ is decomposable. We use induction on the number of indecomposable
direct summands of $M$. By
the assumption  we have $M=\frac{q-1}{q^{x_{st}}-1}M(st)M'$. Therefore
\begin{center}
\begin{tabular}{ll}
$\theta(M)$ & $=\frac{q-1}{q^{x_{st}}-1} \theta(M(st))\theta(M')$\\
& $= \frac{q-1}{q^{x_{st}}-1}\sum_{D\in D_{r-1}}e_{E_{st}+D}e_{A-E_{st}+D'}$,
\end{tabular}
\end{center}
where $D'$ is the diagonal matrix with non-negative integers as entries such that
$\mathrm{co}(E_{st}+D)=\mathrm{ro}(A-E_{st}+D')$.
Suppose that $e_B$ appears in
the multiplication of $e_{E_{st}+D}e_{A-E_{st}+D'}$ and $(f, h)\in \mathcal{O}_B$.
Note that in the minimal projective resolution $Q\rightarrow P$ of $M'$, the projective module
$P_t$ is not a direct summand of $P$. Therefore by the definition of the multiplication
$e_{E_{st}+D}e_{A-E_{st}+D'}$, we have
$f/h\cong M$, that is $B=A$,  and the coefficient
of $e_B$ is about the possibilities of choosing a submodule, which is
isomorphic to $P_s$,  of $P_s^{x_{st}}$. Hence the coefficient is $ \frac{{q^{x_{st}}-1}}{q-1}$
and so $\theta(M)=\sum_{D}e_{A+D}=l_{A, r}$. This finishes the proof.
\end{proof}

\begin{remark}
Proposition 2.3 in \cite{GreenR} has a minor inaccuracy. Indeed,
there is a coefficient missing in front of the image $\theta(M)$.
For example, let $n=3$, $r=2$ and let $M$ be the module determined by the
elementary matrix $E_{13}$. Then
$\theta(M)=vy_{X, r}$, but not $y_{X, r}$, as stated in Proposition 2.3 in \cite{GreenR},
here $\theta$ is the original map
from the quantised enveloping algebra to the Schur algebra $S_v(3, 2)$ and
$y_{X, r}=\sum_{D\in D_1}[X+D]$.
\end{remark}

\subsection{A homomorphism of algebras $\Gamma: \mathbb{Q}\mathcal{M}\rightarrow S_0^+$}
For a given module $M$, denote by $|M|_{\mbox{dir}}$ the number of indecomposable direct
summands of $M$. Let
$\Gamma: \mathbb{Q}\mathcal{M} \rightarrow S_0^+$ be the map given by
\[\Gamma(M)= \left\{\begin{tabular}{ll}
$\theta(M)$ & if  $|M|_{\mbox{dir}}\leq r$, \\ \\
 $0$ & {otherwise}.
\end{tabular}\right.\]

Let $X$ be a module and let $D=\mathrm{diag}(d_1, \cdots, d_n)$ be a diagonal matrix.  Write $X=\oplus_{i, j}M(ij)^{x_{ij}}$.
By $e_{X+D}$ we mean
the basis element in $S_q(n, r)$ corresponding to
the matrix with its entry at $(i, j)$ given by $x_{ij}+\delta_{ij}d_i$ for $i\leq j$ and $0$
elsewhere, where $\delta_{ij}$ are the Kronecker data. Let $\sigma=(\sigma_i)_i: Q\rightarrow P$ be an injection
of $Q$ into $P$, where $P$ and $Q$ are projective modules. Then $(P, Q)$
gives a pair of flags $(f_1, f_2)$ with $f_2$ a subflag of $f_1$. More precisely, the i-th step of $f_1$
is given by $\im P_{\alpha_{n-2}}\cdots  P_{\alpha_i}$ for $i\leq n-2$ and the $(n-1)$-th step is given by the
vector space associated to vetex $n-1$ of $P$, where $P_{\alpha_j}$ is the linear map on
the arrow $\alpha_j$ from $j$ to $j+1$ for the module $P$.
The i-th step of $f_2$ is given by
 $\im P_{\alpha_{n-2}}\cdots P_{\alpha_i}\sigma_i$ for $i\leq n-2$ and the $(n-1)$-th step is given
 by $\im \sigma_{n-1}$.
We have the following result.

\begin{theorem}\label{quotient}
The map $\Gamma$ is a morphism of algebras.
\end{theorem}

\begin{proof}
The unit in $\mathcal{M}$ is the zero module. By the definition of $\Gamma$,
it is clear that $\Gamma(0)= \sum_{D\in D_r}e_D$, the unit of $(\Theta_r^u, \circ)$.

We now need only to show
$$ \Gamma(M\ast N)=\Gamma(M)\circ\Gamma(N). \;\;\; (1)$$

Let $X$ be a generic point $\mathcal{E}(M, N)$. Denote the number of indecomposable
direct summands of $X$, $M$, $N$ by $a$, $b$ and $c$,  respectively. By the definition of
$\Gamma$, we can write
$$\Gamma(X)=\sum_{D\in D_{r-a}}e_{X+D}, $$
$$ \Gamma(M)=\sum_{D'\in D_{r-b}}e_{M+D'},$$
$$ \Gamma(N)=\sum_{D''\in D_{r-c}}e_{N+D''}.$$

We first consider the case where $a>r$. Clearly, in this case $\Gamma(M\ast N)=0$ and we
claim that $e_{M+D'}e_{N+D''}=0$ for any $D'\in D_{r-b}$ and $D''\in D_{r-c}$. In fact suppose that
$e_{L+D'''}$ appears in the multiplication, where $L$ is a module and $D'''$ is a diagonal matrix.
Then $|L|_{\mathrm{dir}}\leq r$, and  $L$ is a degeneration of $X$. Since $Q$ is linearly oriented,
$|L|_{\mbox{dir}}\geq a$.
This is  a contradiction. Therefore $e_{M+D'}e_{N+D''}=0$, and so
$$(\sum_{D'\in D_{r-b}}e_{M+D'})\circ (\sum_{D''\in D_{r-c}}e_{N+D''})=0.$$
This proves the equation (1) for the case $a>r$.

Now suppose that $a\leq r$. Note that if $(D', D'')\not=(C', C'')$, where
$D', C'\in D_{r-b}$ and $D'', C''\in D_{r-c}$, then $e_{M+D'}\circ e_{N+D''}\not=  e_{M+C'}\circ e_{N+C''}$.
By Proposition \ref{genericprop}, we know that if $e_{M+D'}\circ e_{N+D''}\not= 0$, then
$e_{M+D'}\circ e_{N+D''}= e_{X+D}$ for some $D\in D_{r-a}$.

On the other hand, we can show that for any
$e_{X+D}$ appearing in the image of $X$ under $\Gamma$, there exist $D'\in D_{r-b}$ and $D''\in D_{r-c}$
such that $e_{M+D'}\circ e_{N+D''}=e_{X+D}$.  Suppose that
$$\xymatrix{0\ar[r]& Q\ar[r]^\sigma&P\ar[r]^\tau& X\ar[r]&0 }
 $$
is the minimal projective resolution of $X$. Write $D=\mathrm{diag}(d_1, \cdots, d_n)$ and
let $Y$ be the projective module $\oplus_iP_i^{d_i}$. Then the pair of flags in $\mathcal{F}\times \mathcal{F}$,
determined by $(P\oplus Y, Q\oplus Y)$, is in the orbit $\mathcal{O}_{X+D}$.
We have the following diagram where each square commutes and all rows and columns
are short exact sequences,

$$
\xymatrix{
\ker \lambda\ar[d]\ar[r]& Q\oplus Y\ar[d]^{\left(\begin{smallmatrix}\sigma&0\\0& I \end{smallmatrix}\right)}\ar[r]& 0\ar[d]\\
K \ar[r]\ar[d]^\lambda & P\oplus Y \;\; \ar[r]^{\left(\begin{smallmatrix}p\tau&0\end{smallmatrix}\right)}
\ar[d]^{\left(\begin{smallmatrix}\tau&0\end{smallmatrix}\right)}&M,\ar[d]\\
N\ar[r]_i&X\ar[r]_p&M
}
$$
where $I$ is the identity map on $Y$,  $K=\ker\left(\begin{smallmatrix}p\tau&0\end{smallmatrix}\right)$ and
$\lambda= (\tau  0)|_{K}$. Let $\xymatrix{Z'\ar[r]^{I}&Z'}$ be the maximal
contractible piece of the projective resolution
$$
\xymatrix{0\ar[r]&K \ar[r] & P\oplus Y \;\;\ar[r]^{\left(\begin{smallmatrix}p\tau&0\end{smallmatrix}\right)}
&M\ar[r]&0}
$$
of $M$, and let $\xymatrix{Z''\ar[r]^{I}&Z''}$ be
the maximal contractible piece of the projective resolution
$$
\xymatrix{0\ar[r]& \ker \lambda \ar[r]&  K\ar[r] &N\ar[r]&0}
$$
of $N$.
Write $$Z'=\oplus_iP_i^{d_i'} \mbox{ and } Z''=\oplus_iP_i^{d_i''}, $$
and let
$$D'=\mathrm{diag}(d_1', \cdots, d_n') \mbox{ and } D''=\mathrm{diag}(d_1'', \cdots, d_n'').$$
Then $e_{M+D'}\circ e_{N+D''}=e_{X+D}$. Therefore,
$$\sum_{D'\in D_{r-b}}e_{M+D'}\circ \sum_{D''\in D_{r-c}}e_{N+D''}= \sum_{D\in D_{r-a}}e_{X+D}.$$
This proves the equations (1), and so the proof is done.
\end{proof}

The following result is a direct consequence of Theorem \ref{quotient}.

\begin{corollary}
$\ker \Gamma=\mathbb{Q}$-$\mathrm{Span}\{M| |M|_{\mathrm{dir}}> r\}$.
\end{corollary}

\subsection{A geometric realisation of 0-Schur algebras}

\begin{theorem}\label{0schuralgebras}
$S_0^+(n, r)\cong \mathbb{Q}S_0^{++}$ as algebras.
\end{theorem}


\begin{proof}

Denote by $\theta_0$ the specialisation of $\theta$ to $0$, that is,
$\theta_0: {\bf H}_0(Q)\rightarrow S_0(n, r)$. We have
$\ker \theta_0=\mathbb{Q}$-$\mathrm{Span}\{ M| |M|_{\mathrm{dir}}> r\}=\ker \Gamma$, where $\Gamma$ is as in
Theorem \ref{quotient}.  Now the proof
follows from the following commutative diagram.
$$\xymatrix{
\ker \theta_0\ar[r]\ar[d]^\cong & {\bf H}_0(Q)\ar[r]^{\theta_0}\ar[d]^\cong & S_0^+(n, r)\ar@{.>}[d]^\cong\\
\ker \Gamma\ar[r] & \mathcal{M}\ar[r]^\Gamma & S_0^+.
}$$
\end{proof}

As a direct consequence of Theorem \ref{0schuralgebras}, we obtain a multiplicative basis of the positive  part
of $0$-Schur algebras, in the sense that the multiplication of any two basis elements is either a basis element
or zero.

\begin{corollary}
The elements in  $\{l_{A, r}|$ $A$ is an strictly upper triangular matrix in $\bigcup_{s\leq r}\Theta_s^u\}$ form a multiplicative basis of
$S^+_0(n, r)$.
\end{corollary}

Under the map $\Gamma$,  this multiplicative basis
$\{l_{A, r}|$ for any $A\in \bigcup_{s\leq r}\Theta_s^u\}$ is the image of the
multiplicative basis for ${\bf H}_0(Q)$ studied in \cite{Reineke}. By
Theorem 7.2 in \cite{Reineke}, the multiplicative
basis for ${\bf H}_0(Q)$ is the specialisation of Lusztig's canonical
basis for a two-parameter quantization of the universal enveloping algebra of $\mathrm{gl}_n$ given
in \cite{takeuchi}. Thus we can consider the basis $\{l_{A, r}|$ $A$ is a strictly upper triangular matrix in
$\bigcup_{s\leq r}\Theta_s^u\}$
as a subset of a specilization of the canonical
basis.


\vspace{3mm}

{\bf Acknowledgement}: The author would like to thank Steffen K\"{o}nig for helpful discussions.

\bigskip

{\parindent=0cm
Mathematisches Insitut, \\
Universit\"{a}t zu K\"{o}ln,\\
Weyertal 86-90, 50931 K\"{o}ln,\\
Germany.\\
email:xsu@math.uni-koeln.de
}

\begin{thebibliography}{88}

\bibitem{BLM} Beilinson, A. A., Lusztig, G. and MacPherson, R.,
{\it A geometric setting for the quantum deformation of ${\rm GL}\sb n$},
Duke Math. J. 61 (1990), no. 2, 655--677.

\bibitem{DJ} Dipper, R.,  James, G., {\it The $q$-Schur algebra},
Proc. London Math. Soc. (3) 59 (1989), no. 1, 23--50

\bibitem{DipperJames} Dipper, R. and James, G.,
{\it $q$-tensor space and $q$-Weyl modules},
Trans. Amer. Math. Soc. 327 (1991), no. 1, 251--282.

\bibitem{Donkin} Donkin, S., {\it The $q$-Schur algebra}, London Mathematical Society Lecture Note Series, 253. Cambridge University Press, Cambridge, 1998. x+179 pp. ISBN: 0-521-64558-1.

\bibitem{Duj} Du, J.,
{\it A note on quantised Weyl reciprocity at roots of unity},
Algebra Colloq. 2 (1995), no. 4, 363--372.


\bibitem{GreenR} Green, R. M., {\it $q$-Schur algebras as quotients of quantised enveloping algebras},
J. Algebra 185 (1996), no. 3, 660--687.

\bibitem{GreenJ}  Green, J. A., {\it Polynomial representations of ${\rm GL}\sb{n}$}.
Algebra, Carbondale 1980 (Proc. Conf., Southern Illinois Univ., Carbondale, Ill., 1980), pp. 124--140, Lecture Notes in Math., 848, Springer, Berlin, 1981.

\bibitem{JS} Jensen, B. T. and Su, X., {\it Singularities in derived categories}, Manuscripta math. 117, 475-490(2005).


\bibitem{KrobTh} Krob, D. and Thibon, J. Y., {\it Noncommutative symmetric functions IV},
Quantum linear groups and Hecke algebras at $q=0$. J. Algebraic Combin. 6 (1997), no. 4, 339--376.


\bibitem{Reineke} Reineke, M.,
{\it Generic extensions and multiplicative bases of quantum groups at $q=0$},
Represent. Theory 5 (2001), 147--163.


\bibitem{RingelBanach}  Ringel, C. M., {\it Hall algebras},  Topics in algebra, Part 1 (Warsaw, 1988), 433--447, Banach Center Publ., 26, Part 1, PWN, Warsaw, 1990.

\bibitem{RingelInv} Ringel, C. M., {\it Hall algebras and quantum groups}, Invent. Math. 101 (1990), no. 3, 583--591.

\bibitem{Solomon} Solomon, L.,
{\it A decomposition of the group algebra of a finite Coxeter group},
J. Algebra 9 (1968), 220--239.

\bibitem{takeuchi} Takeuchi, M., {\it A two parameter quantization of $\mathrm{Gl}_n$}, Proc. Japan, Acad. 66 (1990),
112-114.

\bibitem{Zwara} Zwara, G., {\it Degenerations of finite-dimensional modules are given by extensions},
Compositio Math. 121 (2000), no. 2, 205--218.

\end{thebibliography}
\end{document}